\documentclass{amsart}

\usepackage{amsthm} 
\usepackage{color}

\usepackage{graphicx,epsfig}
\usepackage{epstopdf}
\usepackage{amsmath, amssymb, latexsym, euscript}
\usepackage{url}
\usepackage[all]{xy}
\usepackage{psfrag}

\definecolor{darkgreen}{rgb}{0.0, 0.5, 0.0}

\newtheorem{theorem}{Theorem}[section]
\newtheorem{lemma}[theorem]{Lemma}
\newtheorem{corollary}[theorem]{Corollary}

\def\gap{\vspace{.3cm}\noindent}

\def\smallskip{\vspace{.15cm}}
\def\medskip{\vspace{.3cm}}
\def\text{\mbox}
\def\rh2{{\mathbb R}{\mathbb H}^2}
\def\ch2{{\mathbb C}{\mathbb H}^2}

\def\RP2{{\mathbb{RP}}^2}
\def\RP3{{\mathbb{RP}}^3}

\def\PGL{\operatorname{PGL}}
\def\H2R{{\mathbb H}^2\times {\mathbb R}}

\def\LC{\mathcal L}
\def\Hom{\operatorname{Hom}}
\def\Pcal{\mathcal P}
\def\Acal{\mathcal A}
\def\dim{\operatorname{dim}}

\newcommand{\bv}{\left[\begin{array}{c}}
\newcommand{\ev}{\end{array}\right]}
\newcommand{\bbmat}{\begin{bmatrix}} 
\newcommand{\ebmat}{\end{bmatrix}}
\newcommand{\bmat}{\begin{matrix}} 
\newcommand{\emat}{\end{matrix}}

\begin{document}
\title{A generalization of the Epstein-Penner construction to projective manifolds.}
\author{D. Cooper, D. D.  Long}
\thanks{Both authors partially supported by grants from the NSF}
\begin{abstract}

 We extend  the canonical cell decomposition due to Epstein and Penner of a  hyperbolic manifold with cusps  to the strictly convex setting.
It follows that a sufficiently small deformation of the holonomy of a finite volume strictly convex real projective manifold is  the holonomy
of some nearby projective structure with radial ends, provided the holonomy of each cusp has a fixed point.
\end{abstract}
\maketitle

One of the powerful constructions in theory of cusped hyperbolic $n$-manifolds is  a cellulation
constructed by Epstein \& Penner in \cite{EP}, which in the particular case that the manifold has one cusp, gives
rise to a canonical cell decomposition. In this note we extend their results to the case of 
strictly convex real projective manifolds.
One consequence is that a small deformation of the holonomy of a finite volume strictly convex
structure on $M$ is  the holonomy of some (possibly not strictly convex) structure on $M$ with {\em radial ends} provided the holonomy
of each cusp has a fixed point in projective space, see  Theorem \ref{deform}.

The proof of   \cite{EP}  employs Minkowski space and shows that if $p$ is a point on the lightcone that corresponds
to a parabolic fixed point then $p$ has a  {\em discrete} orbit. The convex hull of this orbit is an infinite sided polytope 
in Minkowski space that is preserved by the group. The boundary of the quotient of this polytope by the group gives the  cell
decomposition.  This approach uses in an essential way the  quadratic  form 
   $\beta=x_1^2+\cdots + x_n^2-x_{n+1}^2$ that defines $O(n,1)$
 to identify Minkowski space with its dual. This gives a bijection between points 
in the orbit  of $p$ and horoballs that cover the cusp corresponding to $p$. The fact these horoballs are disjoint implies the 
orbit of $p$ is discrete. 

In this paper we use a {\em Vinberg hypersurface} to give a bijection between  the orbit of $p$  and horoballs
in the universal cover of the {\em dual} projective manifold that cover the dual cusp. 

In the hyperbolic case in dimension $2$ one obtains a cell decomposition of moduli
space from the result of Epstein and Penner, \cite{U}. For finite volume
hyperbolic structures Mostow-Prasad rigidity implies that in dimension at least 3  the moduli space is a point.
No similar result holds
in the strictly convex setting: there are examples of one cusped 3-manifolds with families
of finite volume strictly convex projective structure. This paper leads to a decomposition of the moduli space
of such structures, but we do not know if the components of this decomposition are cells.

Background for theory of
cusped projective manifolds can be found in \cite{CLT}. 
A subset $\Omega\subset {\mathbb R}P^n$ is {\em properly convex} if it is the interior  
of a compact convex set $K$ that is disjoint from some codimension-1 projective
hyperplane and {\em strictly convex} if in addition $K$  contains no line segment of positive length
in its boundary. 

A {\em strictly convex real projective $n$-manifold} is 
$M=\Omega/\Gamma$
where $\Omega\subset{\mathbb R}P^n$ is  strictly convex and 
$\Gamma\cong\pi_1M$ is
a discrete group of projective transformations that preserves $\Omega$ and acts freely
on it. We may, and will, lift $\Gamma$ to a subgroup of $SL(\Omega)$ which is the
group of matrices of determinant $\pm1$ that preserve $\Omega$. An element of 
$SL(\Omega)$ is {\em parabolic} if all its eigenvalues have modulus $1$ and it
is not semisimple.

A {\em maximal cusp} in $M=\Omega/\Gamma$ is a connected submanifold, $B$,
such that $\partial B=\overline{M\setminus B}\cap B$ is compact and
\begin{itemize}
\item[C1] Every component $\tilde{B}$ of the pre-image of $B$ in $\Omega$ has strictly convex interior.
\item[C2] $p=cl(\tilde{B})\cap\partial\Omega$ is a single point called
a {\em parabolic fixed point}.
\item[C3] The stabilizer $\Gamma_{\tilde{B}}\subset\Gamma$ of $\tilde{B}$ fixes $p$.
\item[C4] There is a unique projective hyperplane $H\subset{\mathbb R}P^n$ with $p=H\cap\overline{\Omega}$.
\item[C5] Every non-trivial element of $\Gamma_{\tilde{B}}$ is parabolic and preserves  $H$.
 \item[C6]    $\Gamma_{\tilde{B}}$ is conjugate into $PO(n,1)$ so contains  ${\mathbb Z}^{n-1}$ as a subgroup of finite index. 
 \end{itemize}
 
 It is proved in \cite{CLT} that
a strictly convex finite volume real projective manifold has finitely many ends and each is a maximal cusp.

We can identify the domain $\Omega\subset{\mathbb R}P^n$ with a subset $\Omega$  of some affine hyperplane 
 in ${\mathbb R}^{n+1}$.  Then 
 $ {\mathcal C}\Omega=({\mathbb R}_{>0})\cdot\Omega\subset{\mathbb R}^{n+1}$ 
is an open cone based at $0$ and ${\mathbb P}({\mathcal C}\Omega)=\Omega\subset{\mathbb R}P^n$ where 
${\mathbb P}:{\mathbb R}^{n+1}\setminus 0\longrightarrow{\mathbb R}P^n$
 denotes projectivization.
The {\em (positive) lightcone} of $\Omega$
is the cone $\LC={\mathcal C}(\partial\Omega)$. 
It is the subset of the frontier of 
${\mathcal C}\Omega$ obtained by deleting $0$ and 
 ${\mathbb P}(\LC)=\partial\Omega$.
 
\begin{theorem}
\label{epsteinpenner}
Suppose $M=\Omega/\Gamma$ is a strictly convex real projective $n$-manifold 
that contains a maximal cusp $B$, and that  
$p\in\LC$ is a point in the lightcone of $\Omega$ such that
${\mathbb P}(p)\in\partial\Omega$ is the parabolic fixed point of $\pi_1B$. 
Then the $\Gamma$-orbit of   $p$ is a discrete subset of ${\mathbb R}^{n+1}$.
\end{theorem}
This has as an immediate consequence:
\begin{corollary}\label{celldecomp} Suppose $M=\Omega/\Gamma$ 
is a strictly convex real projective $n$-manifold of finite volume
with  at least one (maximal) cusp and ${\mathcal Q}\subset\partial\Omega$ is
the set of  fixed points of parabolics in $\Gamma$.

Then there is a $\Gamma$-invariant decomposition
of $\Omega\cup{\mathcal Q}$ into the $\Gamma$ orbits of finitely many
convex polytopes, each with vertices in ${\mathcal Q}$ and with disjoint interiors. 
The interior of each polytope projects injectively into $M$.
\end{corollary}
%
%
%
%
In the case that the manifold has one cusp, this decomposition becomes canonical:
\begin{corollary}\label{celldecomp_onecusp} Suppose $M=\Omega/\Gamma$ 
is a strictly convex real projective $n$-manifold  that contains a unique (maximal) cusp $B$.

Then there is a canonical decomposition of $M$ into finitely many cells. This decomposition
varies continuously  as the projective structure varies in the sense that the cells in
projective space covering this decomposition
vary continuously.
\end{corollary}
It follows immediately that the isometry group of a one-cusped strictly convex real projective $n$-manifold
is finite. There is an extension of the continuity part of this statement to the multi-cusp case.  

The following definition is essentially due to Choi \cite{choi}.
A submanifold $B$ of a projective $n$-manifold $M$  is a {\em radial end}  if  $M=A\cup B$ with $\partial A=A\cap B=\partial B$
and $B$ is foliated by  rays oriented away from $\partial B$ which develop into oriented lines in ${\mathbb R}P^n$ so that the limit 
of all these lines in the direction given by the orientation
is a single point  $x\in{\mathbb R}P^n$.
A maximal cusp is a radial end. 

There are interesting
examples of cusped hyperbolic 3-manifolds for which the holonomy has many nearby deformations. The following
ensures these correspond to {\em at least one} nearby projective structure with radial ends.

\begin{theorem}\label{deform} Suppose $\rho:[0,\delta)\longrightarrow \Hom(\pi_1M,\PGL(n+1,{\mathbb R}))$ is continuous
and $\rho(0)$ is the holonomy of a finite volume strictly convex real projective structure on the $n$-manifold $M$.
Also assume that for all $t$ that the restriction of $\rho(t)$ to each cusp of $M$ has at least one fixed point in ${\mathbb R}P^n$.
 Then for some $\epsilon\in(0,\delta)$ and for all $t\in[0,\epsilon)$ there is a nearby (possibly not strictly convex) real projective structure 
 with radial ends on $M$
and with holonomy $\rho(t)$.
\end{theorem}

Let  $V$ be a real vector space of dimension $(n+1)$ with dual $V^*$
and ${\mathbb P}:V\setminus 0\longrightarrow{\mathbb P}V$ the projectivization map.
In the following discussion $\Omega$ is a properly 
convex set in ${\mathbb P}V$; we do
not require the extra hypothesis of strictly convex.

The relation between a vector space  and its dual gives rise
to projective duality. Given a properly convex $\Omega\subset{\mathbb P}(V)$
the {\em dual cone} ${\mathcal C}\Omega^* \subset V^*$, 
 is the set of linear functionals which
take strictly positive values on $ {\mathcal C}\Omega$. The {\em dual domain}
$\Omega^*={\mathbb P}({\mathcal C}\Omega^*)\subset {\mathbb P}(V^*)$ is also 
properly convex. If $\Omega$ is strictly convex then so is $\Omega^*$.
The {\em dual lightcone} $\LC^*$ of $\Omega$ is the
lightcone of $\Omega^*$.

The dual action of an element $\gamma\in PGL(V)$ on $V^*$ is given by
 $\gamma^*(\phi)=\phi\circ\gamma^{-1}$. A choice of basis for $V$ 
gives isomorphisms $V\cong{\mathbb R}^{n+1}\cong V^*$ and
$PGL(V)\cong PGL(n+1,{\mathbb R})\cong PGL(V^*)$.
 Using these identifications
the dual
action of $PGL(V)$ on $V^*$ then corresponds to the Cartan involution 
$\theta(A)=(A^{-1})^t$ on $PGL(n+1,{\mathbb R})$. 
If $\Gamma\subset PGL(n+1,{\mathbb R})$ the {\em dual
group} is $\Gamma^*=\theta(\Gamma)$.

The {\em dual manifold} of $M=\Omega/\Gamma$ is  $M^*=\Omega^*/\Gamma^*$.  If $p\in\partial\Omega$ is the parabolic fixed point of a 
maximal cusp $B\subset M$ by (C4) there is a unique supporting hyperplane $H$ to $\Omega$ at $p$. The {\em dual parabolic fixed point} 
$[\phi]\in\partial\Omega^*$ is defined by ${\mathbb P}(\ker \phi)=H$. 
The dual action of $\pi_1B$
fixes $[\phi]$ and there is a {\em dual cusp} $B^*$, well defined up to the equivalence relation generated by inclusion. 
Thus $\phi$ is a point
on the dual light cone. Below we show that level sets of $\phi$ determine
a type of horosphere in $\Omega$ centered at $p$.

The {\em hyperboloid
 model} of hyperbolic space is a certain level set of the quadratic form $\beta$.
In general the holonomy of a strictly convex manifold does not 
preserve any non-degenerate quadratic form but
 it does preserve  a certain convex function which has levels
 sets called {\em Vinberg hypersurfaces} \cite{V} that
 provide a generalization of the hyperboloid. 
 We briefly recall the construction  here.  Let $d\psi$ be a volume form on $
V^*$.  Then the {\em characteristic function} $f:{\mathcal
C}\Omega \longrightarrow {\mathbb R}$ is defined by 
$$f(x)=\int_{{\mathcal C}\Omega^*}e^{-\psi(x)}d\psi$$ 
This is real analytic, convex, and satisfies
$f(tx)=t^{-n}f(x)$ for $ t > 0$.  For each $t>0$ the level set
$S_t=f^{-1}(t)$ is called a {\em Vinberg hypersurface} and is convex.  These sets foliate ${\mathcal
C}\Omega$ and are
permuted by homotheties fixing the origin. For example, the
hyperboloids $z^2=x^2+y^2+t$ are Vinberg hypersurfaces in the cone
$z^2>x^2+y^2$.  The surfaces $S_t$ are all preserved by $SL({\mathcal
C}\Omega)$, and in particular by $\Gamma$. Henceforth, we fix some choice
$S:=S_1$ which we refer to  as {\em the} Vinberg surface
for $\Omega$. It is a substitute for the hyperboloid model
of hyperbolic space.

Let $\pi:S\longrightarrow \Omega$ be the restriction of the projectivization map.
A  point, $\phi\in\LC^*$, in the dual lightcone of $\Omega$  determines a {\em horofunction}
$$h_{\phi}=\phi\circ\pi^{-1}:\Omega\longrightarrow{\mathbb R}$$
Since $\phi\in\LC^*$
it follows that $\ker\phi$ contains a ray $(0,\infty)\cdot v\subset\LC$ in the lightcone.
This function is convex  so the sublevel set ${\mathcal B}(\phi,t)=h_{\phi}^{-1}(0,t]$
 is convex and is called a {\em horoball} associated to $\phi$.
It is in general different from the {\em algebraic horoballs} defined in \cite{CLT}. The
boundary $${\mathcal S}(\phi,t)=\partial {\mathcal B}(\phi,t)=h_{\phi}^{-1}(t)$$ 
of a horoball is called
a {\em horosphere}, and is analytic. 

\begin{lemma}\label{horoball} Suppose $\Omega\subset{\mathbb R}P^n$ is properly convex
 and $\phi$ is a point in the dual lightcone. Then the horofunction $h_{\phi}:\Omega\longrightarrow(0,\infty)$
 is a smooth surjective submersion. Hence for all $t>0$ the horoball 
 ${\mathcal B}(\phi,t)$ is non-empty and convex. 
 \end{lemma}
\begin{proof} The implicit function theorem and smoothness of $f$ imply $h_{\phi}$ is smooth.   
The set $W=f^{-1}(0,1]$ is the subset of ${\mathbb R}^{n+1}$ above $S$ and
is convex, refer to the diagram. Thus  
$X_t=W\cap\phi^{-1}(0,t]$ is convex hence 
${\mathcal B}(\phi,t)=\pi(X_t)\subset\Omega$ is convex. We claim that for $t>0$ it is not
empty. This follows from the fact that if $v$ is a point in the lightcone of $\Omega$
that is also in $\ker\phi$ then
the vertical distance $\delta(t\cdot v)$ between $S$ and $t\cdot v$
goes to zero as $t\to\infty$. 
This is shown by direct computation 
in the particular case that $\Omega$ is an open simplex, and the general result
follows by a comparison argument. Below we give details. 

\begin{figure}[h]
 \begin{center}
\psfrag{b}{$\delta(t\cdot v)$}
\psfrag{S}{$S$}
\psfrag{t}{$X_t$}
\psfrag{v}{$v$}
\psfrag{C}{$\partial({\mathcal C}\Omega$)}
\psfrag{f}{$\phi<t$}
 \includegraphics[scale=1]{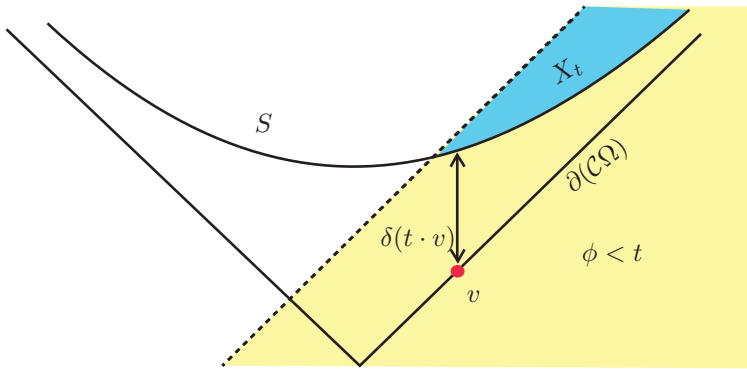}
  \label{VinSurface}\caption{Vinberg hypersurface inside the lightcone}
 \end{center}
 \end{figure}

To begin with, consider the special case that ${\mathcal C}\Omega$ is 
the positive orthant in ${\mathbb R}^{n+1}$,  which  is the cone 
on an $n$-simplex $\Omega=\sigma$.   In an appropriate basis, the group 
$SL({\mathcal C}\sigma)$ contains positive 
diagonal matrices of determinant $1$. The orbits  are the  Vinberg hypersurfaces 
$x_0\cdot x_1\cdots x_n = c$ for each $c > 0$.  Each hypersurface is
asymptotic to the ray $(0,\infty)\cdot v$ which proves the result in this special case.

In the general case  there is a simplex
$\sigma$ with interior in $\Omega$ and $v$ as a vertex. 
Then ${\mathcal C}(\sigma) \subset {\mathcal C}\Omega $, and it follows from the definition that 
$f_{{\mathcal C}(\sigma)} \geq  f_{{\mathcal C}\Omega}|_{{\mathcal C}(\sigma)}$, 
which implies that the Vinberg surface for $\Omega$ lies below that for $\sigma$
along the direction given by $v$.
The claim now follows from the special case.

It follows from this, and the convexity of $S$, that $S$
is never tangent to  a level set of $\ker\phi$ , hence $h_{\phi}$ is a submersion.
\end{proof}

It is immediate from the definition of horoball that if $\gamma\in SL(\Omega)$
and $\phi$ is in the dual light cone then 
$$\gamma({\mathcal B}(\phi,t))={\mathcal B}(\gamma^*\phi,t)$$
Thus if $\gamma^*\phi=\phi$ then the horoball ${\mathcal B}(\phi,t)$
is preserved by $\gamma$. If $B$ is a maximal cusp of $M$ then $\pi_1B$ preserves 
horoballs corresponding to the unique supporting hyperplane at the 
 parabolic fixed point. The next result is that
a sufficiently small such
horoball projects
 to an (embedded) cusp in  $M$ called a {\em horocusp}. Recall that a subset $U\subset\Omega$
 is {\em precisely invariant} under a subgroup $G\subset\pi_1M$ if
 every element of $G$ preserves $U$ and every element of $\pi_1M\setminus G$
 sends $U$ to a subset disjoint form $U$.
\begin{corollary}\label{cuspcor} Suppose  $M=\Omega/\Gamma$ is a properly convex manifold
and  $B\subset M$ is a maximal cusp with holonomy $\Gamma_B\subset\Gamma$ and with dual
 $\Gamma_B^*$ that fixes the dual parabolic fixed point
$[\phi]\in\partial\Omega^*$.

Then there is $t>0$ such that the horoball ${\mathcal B}(\phi,t)$ is precisely invariant
under $\Gamma_B\subset\Gamma$. For sufficiently small  $t>0$ the manifold 
$B'={\mathcal B}(\phi,t)/\Gamma_B$ projects injectively onto a cusp
$B'\subset B\subset M$ and $B\setminus B'$  is bounded.
The horofunction $h_{\phi}$ covers a proper smooth submersion 
$h:B'\longrightarrow (0,t]$. 
The sublevel sets are convex and the level sets, called {\em  horomanifolds}, are compact and give a product foliation.
\end{corollary}
\begin{proof} Let $H$ be the supporting hyperplane to $\Omega$ at the parabolic fixed point for $\Gamma_B$.
By (C3)  $H$ is preserved by $\Gamma_B$.
There is a  codimension-1
subspace $V\subset{\mathbb R}^{n+1}$ with ${\mathbb P}(V)=H$.
 Then $V=\ker\phi$ iff $\phi$ is a point
on the dual lightcone $\LC^*$ such that $[\phi]\in\partial\Omega^*$ is the
 parabolic fixed point for $\Gamma_B^*$.
 
The dual of a parabolic matrix is also parabolic, thus $\phi\in V^*$ is
an eigenvector with eigenvalue $1$ for every element of $\Gamma_B^*$. This means $\phi$ 
is a $\Gamma_B$-invariant function on $V$,
thus $h_{\phi}$
covers a well defined function $h:\Omega/\Gamma_B\longrightarrow(0,\infty)$.  This is a submersion because $h_{\phi}$ is,
hence the level sets ${\mathcal H}_t=h^{-1}(t)$  are a foliation of $B'$. 

Clearly $\pi_1{\mathcal H}_t\cong\pi_1B$ and,
 since $\pi_1B$ is a maximal cusp, by (C6) $\pi_1B$ has (virtual) cohomological dimension $(n-1)$, and it follows that
each horomanifold is compact.  There is a transverse foliation by lines going out into the cusp, and these foliations
give a product structure on $B'$. Convexity of sublevel sets follows from convexity of horoballs.
\end{proof}

\begin{proof}[Proof of \ref{epsteinpenner}] 
We will prove theorem for the dual manifold  $M^*=\Omega^*/\Gamma^*$.
The result follows using the canonical isomorphism between
a finite dimensional vector space and its double dual. As before we
write $V={\mathbb R}^{n+1}$.

Suppose $\phi\in\LC^*$ is a point on the dual light cone fixed by the dual action of $\pi_1B$
and that $\gamma_n^* \in \Gamma^*$ is a sequence such
that $\gamma_n^*\phi=\phi\circ\gamma_n^{-1}$ converges to some $\psi\in  V^*$. 
We must show this
sequence is eventually constant. If not, we may assume $\gamma_n^*\phi$ are all
distinct.
Clearly $\psi\in\LC^*\cup 0$.

By \ref{cuspcor} there is $t>0$ such that $B$ contains the horocusp 
$B'={\mathcal B}(\phi,t)/\pi_1B$.
Choose $w\in{\mathcal C}\Omega$ with $\psi(w)<t$. Since $\gamma_n^*\phi\to\psi$
 then for all $n$ sufficiently large $\gamma_n^*\phi(w)<t$. 
Since $\psi\in\LC^*$ there is  $v\in\LC$ with $\psi(v)=0$ thus 
 $\psi(w+s\cdot v)=\psi(w)$. From the formula for the characteristic function 
 $f$ for $\Omega$ we see that $f(w+s\cdot v)\to 0$ as $s\to\infty$  (also see Figure 1),
  so for $s$ sufficiently
large $w+s\cdot v$ is above the Vinberg hypersurface, which implies
 $w+s\cdot v\in\gamma_n{\mathcal B}(\phi,t)\cap \gamma_{n+1}{\mathcal B}(\phi,t)$.
Since $B'$ is precisely invariant for $\pi_1B\subset\pi_1M$ it follows
that  $\gamma_{n+1}^{-1}\gamma_n\in\pi_1B$. But $\phi$ is preserved
by every element of this group thus 
$\gamma_n^*\phi=\gamma_{n+1}^*\phi$ which is a contradiction.
\end{proof}
\noindent
The next result is well known, but we include it for the convenience of the reader.
\begin{lemma}[orbits are dense]\label{cuspsaredense} Suppose $M=\Omega/\Gamma$ is a strictly convex projective manifold with finite volume.
Then every $\Gamma$ orbit  is dense in $\partial\Omega$.\end{lemma}
\begin{proof} Given a point $b\in\partial\Omega$ define $\Omega^-$ to
 be the intersection with $\Omega$ of the closed convex hull of $\Gamma b$. 
Since $\Omega$ is strictly convex,  $\Gamma b$ is dense in $\partial\Omega$ iff $\Omega=\Omega^-$. The set $\Omega^-$
is convex and $\Gamma$ invariant. The projection, $N=\Omega^-/\Gamma$,  of $\Omega^-$ is a submanifold of $M$. Since $M$ is finite
volume and strictly convex, it is the union of a compact set and finitely many cusps, \cite{CLT}. We replace
$\Omega^-$ by a $K$-neighborhood with $K$ so large that the complement of $N$ is now a subset of the cusps of $M$.

Suppose $R$ is a component of $\Omega\setminus\Omega^-$, with stabilizer
 $\Gamma_R\subset\Gamma$. Then $R/\Gamma_R$
is mapped injectively into a cusp $B\subset M$ by the projection.
Let $\tilde{B}$ be the component of the pre-image of $B$ that contains $R$.
 Observe that $\overline{R}\cap\partial\Omega$ contains more than one point, but  $cl(\tilde B)\cap\partial\Omega$
 is one point by (C2), a contradiction.
\end{proof}

The hypothesis of strictly convex in the above can not be weakened to properly convex because
 there is a properly convex projective torus that is the quotient of the interior of a triangle by a discrete group
 and each vertex of the triangle is an orbit.
 
\begin{lemma}\label{loworbit} Suppose $M=\Omega/\Gamma$ is strictly convex and has finite volume. If
$x\in\LC$ is a point in the light cone then $0$ is  an accumulation
point of $\Gamma x$ iff ${\mathbb P}(x)\in\partial\Omega$ is not a parabolic fixed point.
\end{lemma}
\begin{proof} If ${\mathbb P}(x)$ is a parabolic fixed point the result follows from \ref{epsteinpenner}.
 We adapt the proof of (3.2) in \cite{EP} to show the corresponding result for the projective dual
$M^*=\Omega^*/\Gamma^*$. The result then follows by duality. 
Let $H={\mathbb P}(V)$ be the unique supporting hyperplane to $\Omega$ at ${\mathbb P}(x)$.
Let $\phi\in\LC^*$ be dual to $H$ thus $V=\ker\phi$. This only defines $\phi$ up to scaling.
The manifold $M$ is the union of a compact thick part $K$ and finitely many cusps. Choose a compact
set $\tilde{K}$ in the Vinberg hypersurface $S$ for $\Omega$ so that the projection of $\tilde{K}$ contains $K$.
Given $t>0$ there is a horoball ${\mathcal B}={\mathcal B}(\phi,t)\subset\Omega$.
If the $\Gamma^*$ orbit of $\phi$ does not accumulate on $0$ we may choose $t$ so small 
($\Rightarrow{\mathcal B}$ small)  that the $\Gamma$ orbit
of ${\mathcal B}$ is disjoint from $\tilde{K}$. This implies the orbit of ${\mathcal B}$ projects into a cusp of $M$. It follows that
$\phi$ is a parabolic fixed point of $\Gamma^*$.
\end{proof}

\begin{proof}[Proof of  \ref{celldecomp}]  Choose a $\Gamma$-invariant collection ${\mathcal P}\subset\LC$ of points in the light cone, 
one in the direction of each parabolic fixed point. If there are $k$ cusps this amounts to choosing $k$
positive reals. By  \ref{epsteinpenner}, ${\mathcal P}$ is a discrete set. The closed convex hull of ${\mathcal P}$ is
 a $\Gamma$-invariant set $C$
 in ${\mathbb R}^{n+1}$. We show that $C$ has polyhedral boundary. The image under projectivization  decomposes  $M$
 into convex cells. Clearly this decomposition is unchanged by uniformly scaling ${\mathcal P}$, and so the $k$ cusps
 result in a family of cell decompositions parameterized by a point in the interior of a simplex in ${\mathbb R}P^{k-1}$.
 
 Regard ${\mathbb R}^{n+1}$ as an affine patch in ${\mathbb R}P^{n+1}={\mathbb R}^{n+1}\sqcup{\mathbb R}P^n_{\infty}$.
The image of $\Omega$ in ${\mathbb R}P^n_{\infty}$ under radial projection from $0\in{\mathbb R}^{n+1}$ 
is a set $\Omega_{\infty}\subset{\mathbb R}P^n_{\infty}$ projectively equivalent to $\Omega$. Define 
$K\subset{\mathbb R}P^{n+1}$ to be the cone that
is the closure in ${\mathbb R}P^{n+1}$  of ${\mathcal C}\Omega$ then 
$$K={\mathcal C}\Omega\sqcup\LC\sqcup\overline{\Omega}_{\infty}\sqcup0$$ It is disjoint
from a codimension-1 projective hyperplane $H$ that is a small perturbation of ${\mathbb R}P^n_{\infty}$
and therefore $K\subset {\mathbb A}^{n+1}={\mathbb R}P^{n+1}\setminus H$ is a properly convex set. 

Observe that ${\mathcal P}\subset\partial K$ and, since $K$ is convex, $\overline{C}\subset K$
where $\overline{C}$ is the closure of $C$ in ${\mathbb R}P^{n+1}$.
 By \ref{cuspsaredense}
it follows that the set of accumulation points of ${\mathcal P}$ is $\partial\Omega_{\infty}$. This is because every open set 
$U\subset \partial\Omega_{\infty}$ contains infinitely many parabolic fixed points. These correspond to
an infinite subset of ${\mathcal P}$.  By \ref{epsteinpenner} all but finitely many of these points are very high in the lightcone
and thus very close to $U$.
This in turn
implies $\overline{C}$ contains $\partial\Omega_{\infty}$ and
therefore also contains $\overline{\Omega}_{\infty}$ thus $$\overline{C}=C\sqcup\overline{\Omega}_{\infty}$$
It follows that
that $\overline{C}$ is the closed convex hull in ${\mathbb A}^{n+1}$ of ${\mathcal P}$.

The holonomy, $\Gamma$, of $M$ lies in $SL(n+1,{\mathbb R})$ and can be identified with a subgroup 
$\Gamma^+\subset \PGL(n+2,{\mathbb R})$
that preserves $K$.  In suitable coordinates $\Gamma^+$ is block diagonal with a trivial block of
size $1$ and the other block is $\Gamma$.

Closely following \S 3 of  \cite{EP} we establish the following claims: 
\\[\baselineskip]
$\bullet$ The dimension of $C$ is $n+1$
because $\overline{C}$ contains the $n$-dimensional set
 $\Omega_{\infty}\subset{\mathbb R}P^n_{\infty}$ and also contains the points
${\mathcal P}$ which are not in ${\mathbb R}P^n_{\infty}$.
\\[\baselineskip]
$\bullet$  $w\in C \cap \LC$ iff $w=\alpha z$ for some $z\in{\mathcal P}$ and $\alpha \geq 1$. \\
\indent If $w$ is not of this form  then the segment $[0,w]$ is disjoint from ${\mathcal P}$. Since ${\mathcal P}$ is discrete there is a small
neighborhood  $U\subset {\mathbb R}^{n+1}$ of this segment that contains no point of ${\mathcal P}$. Hence there is a hyperplane
that intersects ${\mathcal C}\overline{\Omega}$ in a small, 
convex, codimension-1 set in $U$ and separates $[0,w]$ from ${\mathcal P}$, and hence from
 $C$. This means $w\notin C$.

For the converse, given $z\in{\mathcal P}$   the image $w\in\partial\Omega_{\infty}$ of $z$ is
in $\overline{C}$, hence $[z, w] \subset \overline{C}$. This contains all the points $\alpha z$ with $\alpha\ge 1$.\\[\baselineskip]
$\bullet$ Each ray $\lambda\subset{\mathcal C}\Omega$ that starts at $0$ meets $\partial C$ exactly once.\\
\indent Since ${\mathcal P}$ is discrete in ${\mathbb R}^{n+1}$ it follows that $0\notin \overline{C}$ so $\lambda$
starts  outside $\overline{C}$ and limits on $q\in\Omega_{\infty}\subset {\overline C}$. Thus $\lambda$ contains points in the interior
of $\overline{C}$. 
Since $\overline{C}$ is convex
$\lambda$ meets $\partial\overline{C}$
in a single point $z$. Since $\lambda\subset{\mathcal C}\Omega$ it follows that  
$z\in {\mathcal C}\Omega\cap \partial\overline C=\partial C$.
\\[\baselineskip]
%
%
%
%
%
%
%
$\bullet$ If $W\subset{\mathbb R}^{n+1}$ is a supporting affine hyperplane for $C$ at a point 
$z \in \partial C \cap{\mathcal C}\Omega$ then $W\cap{\mathcal C}(\overline{\Omega})$ is compact and convex.\\
\indent The closure $\overline{W}$ of $W$ in ${\mathbb R}P^{n+1}$ is a projective hyperplane that is a supporting hyperplane
for $\overline{C}$ in ${\mathbb R}P^{n+1}$. Clearly $\overline{W}$ is disjoint from $\Omega_{\infty}$ and
by the previous claim $0\notin W$.  The ray from $0$ through $z$ limits on $\Omega_{\infty}$
and crosses $\partial C$ at $z$ therefore  $\overline{W}$
separates $0$ from $\Omega_{\infty}$ in ${\mathbb A}^{n+1}$

Let $V$ be the vector subspace parallel to $W$. Then $V=\ker\phi$ for some linear map $\phi$.
We claim that $V$ is disjoint from ${\mathcal C}\overline{\Omega}={\mathcal C}\Omega\sqcup\LC$. 
Observe that $\overline{V}$ and $\overline{W}$ have the same intersection with ${\mathbb R}P^n_{\infty}$
and are disjoint from $\Omega_{\infty}$. Since $V$ contains $0$ it follows that $V$ is disjoint from ${\mathcal C}\Omega$.
It remains to show $V$ is disjoint from $\LC$.

Define an affine function $\psi$
on $ {\mathbb R}^{n+1}$ by $\psi(v) = \phi(v-z)$. Then $W=\psi^{-1}(0)$, and since $W$ is a supporting 
hyperplane for $C$, this means that 
$\psi$
has constant sign on $C$. By replacing $\phi$ by $-\phi$ if needed we may assume $\psi(v)\ge 0$ for all $v\in C$.
Hence $\phi(v)\ge\phi(z)$ for all $v\in C$. Since $V$ is disjoint from ${\mathcal C}\Omega$ it follows that $\phi$ has
constant sign on ${\mathcal C}\Omega$. Since $\psi$ takes arbitrarily large positive values on ${\mathcal C}\Omega$
it follows that $\phi\ge0$ on ${\mathcal C}\Omega$ and hence $K=\phi(z)>0$. This implies
$\phi(v)\ge K$ for all $v\in C$. Since $\Gamma$ preserves $C$ 
it follows that for every $\gamma^*\in\Gamma^*$ that $\gamma^*\phi\ge K$ everywhere
on $C$.
%
%
%
%
%
%

We claim that  $\LC\cap V = \emptyset$. For, suppose not and that $0 \neq x\in\LC\cap V$. 

Firstly, observe that  ${\mathbb P}(x)$ cannot be a parabolic fixed point, otherwise points high on this ray are 
in $C$. However, $V$ and $W$ are parallel hyperplanes, so that  these high points, which are all in $V$,
 must be below $W$, a  contradiction.

However, consideration of stabilizers now implies that ${\mathbb P}(\phi)\in\partial\Omega^*$ is not a (dual) parabolic 
fixed point. Hence  by \ref{loworbit}, there is a sequence $\gamma_k^*\in\Gamma^*$ such that $\gamma_k^*\phi\to 0$.
Thus for large $k$ we have $\gamma_k^*\phi(z)<K$. This contradicts $\gamma_k^*\phi\ge K$ everywhere on $C$. 

This proves the assertion  that $\LC\cap V =  \emptyset$. Our main claim that  $W\cap{\mathcal C}(\overline{\Omega})$ is 
compact and convex now follows:  It is clear that this set is convex. Note that $\phi(W) = \phi(V+z) = \phi(z) = K$ is constant. 
However, since  
$\LC\cap V =  \emptyset$, for any ray $({\mathbb R}_{>0})\cdot v$ in ${\mathcal C}\overline{\Omega}$, $\phi(v) > 0$, so that very high points
on that ray take values $> K$. It follows that $W$ meets ${\mathcal C}\overline{\Omega}$ in a compact set. \\[\baselineskip]
$\bullet$ Every point in $\partial C\cap{\mathcal C}\Omega$ is contained in a supporting hyperplane
that contains at least $(n+1)$
points in ${\mathcal P}$.\\
 \indent Given a supporting hyperplane $H$, rotate it around $H\cap C$ until
 it meets another point of ${\mathcal P}$. Since this set is discrete, there is a first rotation angle with this property.
 This process stops when $H\cap C$ contains an open subset of $H$.
 See \cite{EP} for more details. \\[\baselineskip]
$\bullet$ The set of codimension-$1$ faces is locally finite inside ${\mathcal C}\Omega$.\\
\indent Let $K \subset {\mathcal C}\Omega$ be a compact set meeting faces $F_1, F_2, .......$ and suppose that these faces are defined by affine hyperplanes $A_1, A_2,.....$. 
Pick $x_i \in K \cap F_i$ and subconverge so that $x_i \rightarrow x$ and $A_i \rightarrow A$, an affine plane containing $x$. 
The $A_i$'s are all support planes, whence so is $A$, thus it meets ${\mathcal C}\overline{\Omega}$
in a compact convex set. Move $A$ upwards a small distance to obtain $A^+$. 
Then all but finitely many of $A_i\cap{\mathcal C}\overline{\Omega}$
lie below $A^+\cap{\mathcal C}\overline{\Omega}$. Hence  ${\mathcal P}\cap\left(\cup A_i\right)$ is finite 
and it follows that  there were only a finite number of faces meeting $K$.\\[\baselineskip]
The locally finite cell structure on $\partial C \cap {\mathcal C}\Omega$ is $\Gamma$-equivariant and  projects to a locally finite cell structure
on $M = \Omega/\Gamma$. This completes the proof of  \ref{celldecomp}.\end{proof}
\begin{proof}[Proof of  \ref{celldecomp_onecusp}] In the case that $M$ has only one cusp, the convex hull $C$ is defined by the orbit of a single vector,
which in turn is uniquely defined up to scaling. It follows that $C$ is defined up to homothety and this is invisible when one projects
into $\Omega/\Gamma$. The fact that the decomposition varies continuously
follows from the discussion below. 
\end{proof}
\begin{proof}[Proof of \ref{deform}] If $M$ is compact, this is well known (for example \cite{G1} and \cite{CEG}), so we may assume $M$ has cusps.
The holonomy of a cusp has a unique fixed point. It follows that a sufficiently small deformation of the holonomy
of a cusp has at most finitely many fixed points. The hypothesis then ensures that for a sufficiently small deformation
each end has at least one isolated fixed point.

By \ref{celldecomp} at $t=0$ there  is a compact convex polytope $P_0\subset{\mathbb R}P^n$
 with face pairings so that the quotient $X_0=P_0/\sim$ is the compactification of the manifold
 $M$ obtained by adding an ideal point for each cusp. One may regard $X_0$ as a projective manifold with a finitely many singular points.
 The vertices of $P_0$ are the parabolic fixed points ${\Pcal}_0=\{p_1(0),\cdots,p_k(0)\}$ of a finite set of (conjugates of) cusp subgroups  $G_1,\cdots,G_k\subset \pi_1M$.
 
 Each face $A_0$ of $P_0$ is a convex polytope that is the convex hull of a subset of $\Pcal_0$. In general this face
 need not be a simplex. Even if it is a simplex, it is possible  that $|A_0\cap \Pcal_0|>1+\dim(A_0)$. 
 Each face $A_0$ can be triangulated using $0$-simplices $A_0\cap \Pcal_0$
 in a way that respects the face pairings. This involves some arbitrary choices, for example if $A_0$ is a quadrilateral 
 one chooses a diagonal.
 
 From now on we regard $P_0$ as a triangulated convex
 polytope, with one vertex, $x_0$, in the interior of $P_0$ that is coned to the simplices in $\partial P_0$.
  Each face of $P_0$ (= simplex in $\partial P_0$) is the convex hull of a subset of $\Pcal_0$.
  Adjacent  codimension-1  faces might lie in the same hyperplane. Moreover the faces of $P_0$ are paired by 
  projective maps, and the identification space is $X_0$.

 By assumption  $\rho_t(G_i)$ has at least one isolated fixed point $p_i(t)$ and by continuity it is
  close to $p_i(0)$.  For each cusp $B\subset M$
 choose a $\rho_t(\pi_1M)$-orbit of isolated fixed point for  $\rho_t(\pi_1B)$. This is one choice per cusp of $M$; if $M$ has one cusp then
 this is a single choice. Let $\Pcal_t$ be the set of chosen fixed points for $\rho_t(G_i)$ for $1\le i\le k$. There is a natural bijection
 $h_t:\Pcal_0\longrightarrow \Pcal_t$. 
 
  For $t$ sufficiently small
 this choice of  fixed points and $x_0$ determines a (possibly non-convex) 
 triangulated polytope $P_t\subset{\mathbb R}P^n$ close to $P_0$ and with the same combinatorics. A face $A_0$ of $P_0$
 is the convex hull of a subset $\Acal_0\subset\Pcal_0$.
 Define  $A_t$ as the convex
 hull of the subset $\Acal_t=h_t(\Acal_0)\subset \Pcal_t$. For $t$ sufficiently small
 $A_t$ is a simplex and the union of these simplices
  is the boundary of $P_t$. Moreover, $\partial P_t$ is a simplicial complex and $h_t$ extends to a
  simplicially isomorphism from $P_0$ to $P_t$, thus $P_t$ is a cell.
 
 We claim there  are  face pairings for $P_t$ close to those of $P_0$. 
 The reason is the following. Suppose $A_0$ and $B_0$ are two faces of $P_0$ and $\rho_0(g)[A_0]=B_0$ for some $g\in\pi_1M$. 
 For each vertex $p=p_i(0)$ of $A_0$ the vertex $(\rho_0g)(p)$ of $B_0$ is in the same orbit as $p$. The vertex $p(t)=p_i(t)$ of $A_t$
 is sent by $\rho_tg$ to the vertex  $(\rho_tg)(p(t))$ of $B_t$ because our choices of isolated fixed points are preserved by
 the action of $\rho_t(\pi_1M)$. It follows that $\rho_t(g)$ sends $A_t$ to $B_t$.
 
 The quotient gives a nearby 
 singular structure on $P_t/\sim_t$, and by deleting the vertices of $P_t$ a nearby projective structure on $M$.
 Moreover, it is clear from this description that the deformed manifold has radial ends.
 \end{proof}
 In fact, it is shown in \cite{CLTkos}, that if in addition   $\rho_t$ satisfies certain conditions in each cusp, then
the nearby structure is properly (or even strictly) convex.\\[\baselineskip]
\noindent
{\bf Remarks.}  The situation in the properly convex case is more involved; the authors hope to explore this in future work. There
are other directions one might explore, for example \cite{Koj} uses similar methods for hyperbolic 3-manifolds with totally geodesic boundary.

\small
\bibliography{EPref.bib} 
\bibliographystyle{abbrv} 

\gap

\address{Department of Mathematics, University of California Santa Barbara, CA 93106, USA}

\email{cooper@math.ucsb.edu}

\email{long@math.ucsb.edu} 

\end{document}